\documentclass[preprint,12pt]{elsarticle}
\usepackage{amsthm,amsmath,amssymb}
\usepackage{graphicx}
\usepackage[colorlinks=true,citecolor=black,linkcolor=black,urlcolor=blue]{hyperref}
\usepackage{mathrsfs}

\newtheorem{thm}{Theorem}[section]
\newtheorem{Lemma}[thm]{Lemma}
\newtheorem{cor}[thm]{Corollary}

\newtheorem{defn}[thm]{Definition}

\newtheorem{remark}[thm]{Remark}

\journal{}

\begin{document}

\title{Characterizations of Eulerian and even-face partial duals of ribbon graphs}
\author{Qingying Deng, Xian'an Jin\\
\small School of Mathematical Sciences\\[-0.8ex]
\small Xiamen University\\[-0.8ex]
\small P. R. China\\
\small\tt Email:xajin@xmu.edu.cn
}
\begin{abstract}
Huggett and Moffatt characterized all bipartite partial duals of a plane graph in terms of all-crossing directions of its medial graph. Then Metsidik and Jin characterized all Eulerian partial duals of a plane graph in terms of semi-crossing directions of its medial graph. Plane graphs are ribbon graphs with genus 0. In this paper, we shall first extend Huggett and Moffatt's result to any orientable ribbon graph and provide an example to show that it is not true for non-orientable ribbon graphs. Then we characterize all Eulerian partial duals of any ribbon graph in terms of crossing-total directions of its medial graph, which are much more simple than semi-crossing directions.
\end{abstract}
\maketitle

\section{Introduction and statements of results}
\noindent

A graph is said to be Eulerian if the degree of each of its vertices is even. A graph is said to be bipartite if it does not contain cycles of odd lengths.
A ribbon graph is a new form of the old cellularly embedded graph and the signed rotation system, which makes the concept of partial duality possible. We do not require graphs to be connected throughout the whole paper unless otherwise stated. The genus of a ribbon graph is defined to be the sum of genuses of its connected components. A plane graph is a ribbon graph with genus 0.

The geometric dual, $G^*$, of a cellularly embedded graph $G$ is a fundamental concept in graph theory and also appeared in other branches of mathematics.
To unify several Thistlethwaite's theorems \cite{Chmu,Chmu2,Lin1} on relations of Jones polynomial of virtual links in knot theory and the topological Tutte polynomial of ribbon graphs constructed from virtual links in different ways, S. Chmutov, in \cite{16},
introduced the concept of partial dual of a cellularly embedded graph in terms of ribbon graph. Roughly
speaking, a partial dual is obtained by forming the geometric dual with respect to only a subset of edges of a cellularly embedded graph (for details, see Subsection \ref{spd}).

Let $G=(V,E)$ be a ribbon graph and $A\subset E$, we denote by $G^A$ the partial dual of $G$ with respect to $A$. Note that $G^E=G^*$, $G$ and $G^*$ as surfaces have the same orientability and genuses. Although $G^A$ has the same orientablity as $G$ but the genus may change. For example, a partial dual of a plane graph may not be plane again. The following theorem is well-known, see, for example, 10.2.10 of \cite{Bon}.

\begin{thm}
Let $G$ be a plane graph. Then $G$ is Eulerian if and only if $G^*$ is bipartite.
\end{thm}

In \cite{HM}, Huggett and Moffatt extended the above connection between Eulerian and bipartite plane graphs from geometrical duality to partial duality.
Note that a bipartite ribbon graph is always an even-face one but the converse is not always true but holds for plane graphs. We shall use even-face to replace bipartite and then we shall have: for a ribbon graph $G$, $G$ is Eulerian if and only if $G^*$ is an even-face graph. Extend it to partial duality, we obtain:

\begin{thm}\label{mainn}
Let $G$ be a ribbon graph, $A\subset E(G)$ and $A^c=E(G)-A$. Then $G^{A}$ is an even-face graph
if and only if $G^{A^c}$ is Eulerian.
\end{thm}

In \cite{HM}, Huggett and Moffatt further determined which subsets of edges in a plane graph give rise to bipartite partial duals. In this paper, we extend their result  from plane graphs to orientable ribbon graphs and obtain the first main result.

\begin{thm}\label{main00}
Let $G$ be an orientable ribbon graph and $A\subset E(G)$. Then $G^{A}$ is bipartite
if and only if $A$ is the set of $c$-edges arising from an all-crossing direction
of $G_m$, the medial graph of $G$.
\end{thm}

In \cite{HM}, Huggett and Moffatt also found a sufficient condition for a subset of edges to give rise to an Eulerian
partial dual in terms of all-crossing directions of medial graphs. In \cite{MJ}, Metsidik and Jin characterized all Eulerian partial duals of a plane graph in terms of semi-crossing directions of its medial graph. In this paper, we find more simple directions, i.e. crossing-total directions, of medial graph to characterize all Eulerian partial duals of any ribbon graph and obtain our second main result.

\begin{thm}\label{main12}
Let $G$ be a ribbon graph and $A\subset E(G)$. Then $G^{A}$ is Eulerian
if and only if $A$ is union of the set of all $d$-edges and the set of some $t$-edges arising from a crossing-total direction
of $G_m$.
\end{thm}

For details of crossing-total directions, see Subsection 3.1. For terminologies and notations not defined in this paper, we refer the reader to \cite{Bon}.

\section{Ribbon graphs and their partial duals}
\noindent

The majority of the content of this section can be found in the monograph \cite{Mo}. Besides basic definitions, we also present some preliminary old or new results.

\subsection{Cellularly embedded graphs and ribbon graphs}

Let $\Sigma$ be a closed surface. A cellularly embedded graph $G=(V(G),E(G))$ $\subset$ $\Sigma$ is a graph drawn on $\Sigma$
in such a way that edges only intersect at their ends and such that each connected
component of $\Sigma \backslash G$ is homeomorphic to an open disc. If $G$ is cellularly
embedded in $\Sigma$, the connected components of $\Sigma \backslash G$ along with their boundaries in $\Sigma$ are called faces of $G$.
Two cellularly embedded graphs $G \subset \Sigma$ and $G' \subset \Sigma'$ are equivalent, written $G=G'$, if there is
a homeomorphism (which is orientation preserving when $\Sigma$ is orientable) from $\Sigma$ to $\Sigma$ that sends $G$ to $G'$. We shall consider cellularly embedded graphs under equivalence.

There are several concepts which are equivalent to cellularly embedded graphs. Among them, ribbon graphs have advantages over others, and one advantage is to define partial duality.

\begin{defn}\label{de1}
A ribbon graph $G=(V(G),E(G))$ is a (possibly nonorientable)
surface with boundary represented as the union of two sets of
topological discs: a set $V(G)$ of vertices and a set $E(G)$ of edges such that:
\begin{itemize}
\item the vertices and edges intersect in disjoint line segments;
\item each such line segment lies on the boundary of precisely one vertex and
precisely one edge;
\item every edge contains exactly two such line segments.
\end{itemize}
\end{defn}

We shall call these line segments in Definition \ref{de1} common line segments \cite{Metrose} since they belong to boundaries of both vertex discs and edge discs.

Intuitively, a ribbon graph arises naturally from a small neighbourhood
of a cellularly embedded graph $G$ in $\Sigma$. Neighbourhoods of vertices of $G$ form the vertices of the ribbon graph,
and neighbourhoods of the edges of $G$ form the edges of the ribbon graph. Conversely, if $G$ is a ribbon graph, we simply sew discs (corresponding to faces) into boundary components of the ribbon graph to get a closed surface $\Sigma$ and the core of $G$ is cellularly embedded in $\Sigma$. Two ribbon graphs are equivalent if they define equivalent cellularly
embedded graphs. Since ribbon graphs and cellularly embedded graphs are equivalent we can, and will, move freely between them.

Let $G$ be (an abstract or a ribbon) graph and $v\in V(G)$. We denote by $d_G(v)$ the degree of $v$ in $G$, i.e. the number of half-edges incident with $v$.

\begin{defn}\cite{Metrose}
Let $G$ be ribbon graph, $v\in V(G)$ and $e\in E(G)$. By deleting the common line segments from the boundary of $v$, we obtain $d_G(v)$ disjoint line segments, we shall call them vertex line segments. By deleting common line segments from the boundary of $e$, we obtain two disjoint line segments, we shall call them edge line segments.
\end{defn}

On a boundary component of a ribbon graph, vertex line segments and edge line segments alternate. The number of vertex line segments (equivalently, edge line segments) on a boundary component of a ribbon graph is called the degree of the boundary component. As an example, in Figure \ref{ls}a, all $v_{ij}$'s are vertex line segments, all $e_{ij}$'s are edge line segments and all $e'_{ij}$'s are common line segments.
$v_{11}e_{32}v_{23}e_{21}v_{31}e_{12}v_{22}e_{31}v_{12}e_{52}$ is a boundary component of degree 5. A ribbon graph is said to be even-face if each of its boundary components is of even degree. Note that the degree of a face of a cellularly embedded graph $G$ is exactly the degree of the corresponding boundary component of the ribbon graph equivalent to $G$. A ribbon graph is even-face if and only if its corresponding cellularly embedded graph is even-face, i.e. each face is of even degree.
The following lemma is obvious.

\begin{Lemma}\label{111}
Let $G$ be a ribbon graph. If $G$ is bipartite, then $G$ is even-face graph.
\end{Lemma}

However, the inverse of Lemma \ref{111} is not true. For example, the graph consisting of longitude and latitude circles in a torus (one vertex and two loops) is even-face, but not bipartite.
Sometimes it is more convenient to describe ribbon graphs using Chmutov's arrow
presentations from \cite{16}.
\begin{defn}
An arrow presentation consists of a set of circles with pairs of labelled arrows, called marking arrows, on them such that there are exactly two marking arrows of each label.
\end{defn}

Note that circles correspond to boundaries of vertex discs and each pair of marked arrows with the same label corresponds to an edge (directions of two arrows are consistent with the orientation of the boundary of the edge).
An example of a ribbon graph and its arrow representation is given in Figure \ref{rib}. Arrow presentations are equivalent if they describe equivalent
ribbon graphs and are considered up to this equivalence. For example, reversing the direction
of both arrows with a given label yields an equivalent arrow presentation.

\begin{figure}[!htbp]
\centering
\includegraphics[width=4.0in]{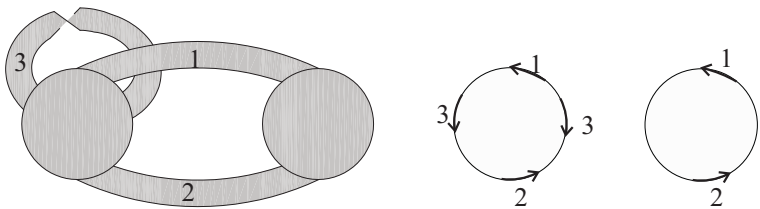}
\caption{{\footnotesize A ribbon graph and its arrow presentation.}}\label{rib}
\end{figure}

A ribbon graph $G$ can be recovered from an arrow presentation by identifying
each circle with the boundary of a disc (forming the vertex of $G$). Then, for each
pair of $e$-labeled arrows, take a disc (which will form an edge of $G$), orient its boundary,
place two disjoint arrows on its boundary that point in the direction of the orientation,
and identify each of these with an $e$-labeled arrow.

We emphasise that the circles in an arrow presentation are not equipped with any
embedding in the plane or $\mathbb{R}^3$.

\subsection{Partial dual}\label{spd}
\noindent

\begin{defn}\cite{16}\label{dpd}
Let $G$ be a ribbon graph and $A\subset E(G)$. Arbitrarily
orient and label each of the edges of $G$ (the orientation need not extend to an
orientation of the ribbon graph). The boundary components of the spanning ribbon
subgraph $(V(G),A)$ of $G$ meet the edges of $G$ in disjoint arcs (where the spanning
ribbon subgraph is naturally embedded in $G$). On each of these arcs, place an
arrow which points in the direction of the orientation of the edge boundary and
is labelled by the edge it meets. The resulting marked boundary components of the
spanning ribbon subgraph $(V(G),A)$ define an arrow presentation. The ribbon graph
arising from this arrow presentation is the partial dual $G^A$ of $G$ with respect to $A$.
\end{defn}

Figure \ref{pd} provides an example of the construction of a partial dual
$G^{A}$ of a ribbon graph $G$ in Figure \ref{pd} (a) with $A=\{1,2,5\}$ using Definition \ref{dpd}. Figure \ref{pd} (b) shows the spanning subgraph $(V(G),\{1,2,5\})$ with boundary
components marked. The arrow presentation of $G^{A}$ is shown in Figure \ref{pd} (c). The ribbon graph of the partial
dual $G^{A}$ is shown in Figure \ref{pd} (d).

\begin{figure}
\centering
\begin{minipage}{7cm}
\includegraphics[width=2.5in]{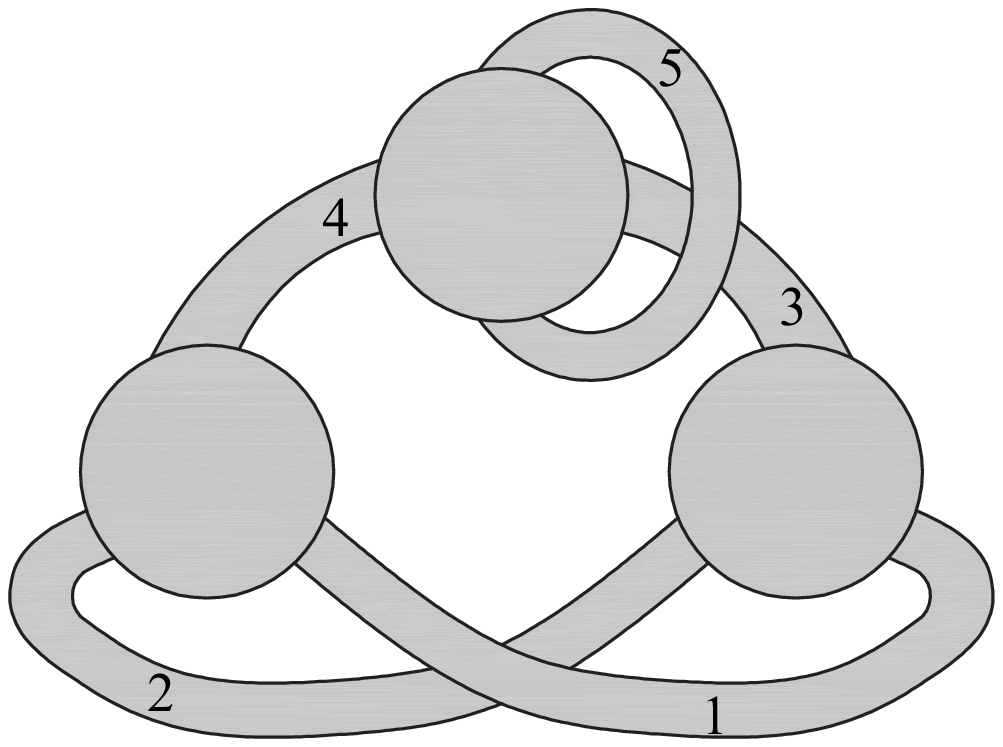}
\centerline{(a)}
\end{minipage}\begin{minipage}{7cm}
\includegraphics[width=2.5in]{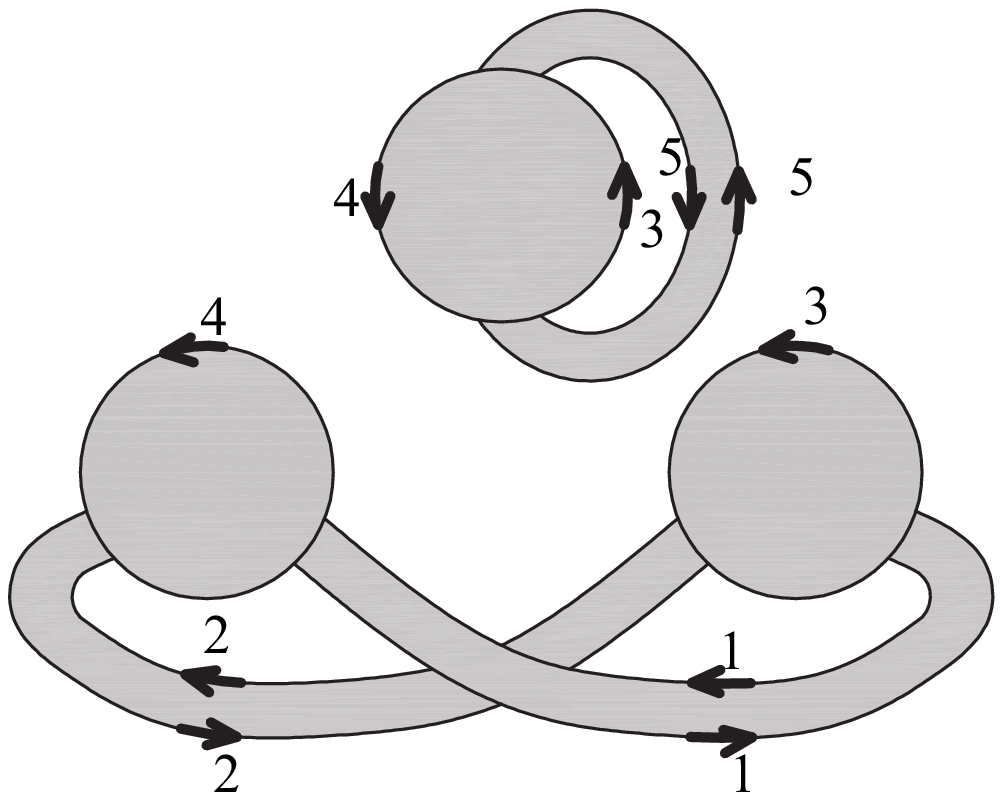}
\centerline{(b)}
\end{minipage}
\begin{minipage}{7cm}
\includegraphics[width=2.5in]{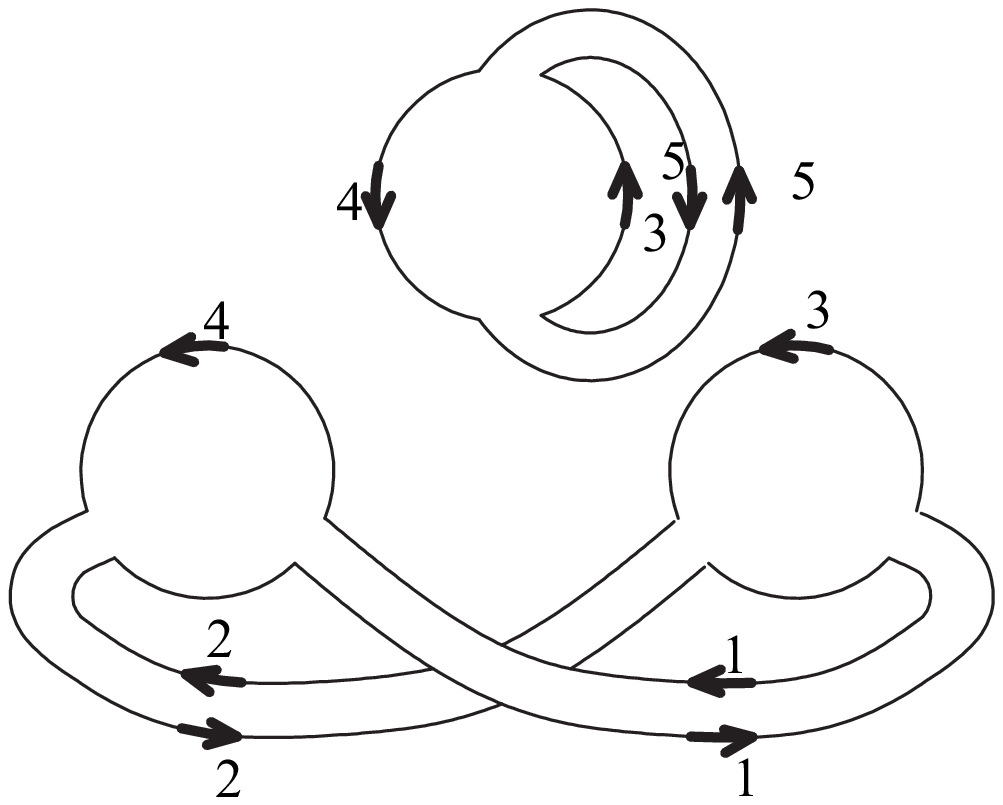}
\centerline{(c)}
\end{minipage}\begin{minipage}{7cm}
\includegraphics[width=2.5in]{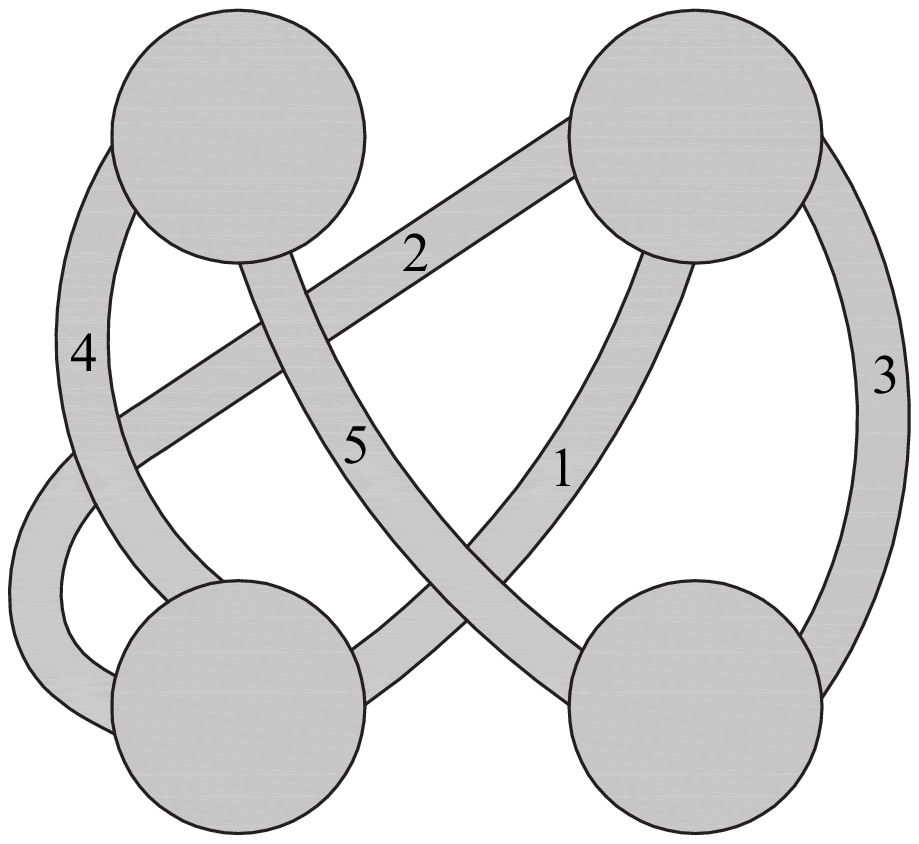}
\centerline{(d)}
\end{minipage}
\caption{An example of the construction of a partial dual: (a) a ribbon graph $G$, (b) the marked spanning subgraph $(V(G),A)$ with $A=\{1,2,5\}$, (c) the arrow presentation of $G^A$, and (d) the ribbon graph form of $G^{A}$.}\label{pd}
\end{figure}

\begin{Lemma}\label{pro1}
Let $G$ be an orientable ribbon graph. Then $G$ has an arrow presentation such that all marking arrows in the same circle have the same directions.
\end{Lemma}
\begin{proof}
When $G$ is orientable, we can draw it in the plane such that all vertex discs and edge discs have no twists. By Definition \ref{dpd} we can get an arrow presentation such that all marking arrows in the same circle have the same directions.
\end{proof}

Form the viewpoint of surfaces, the partial dual $G^A$ can be obtained from $G$ by gluing a disc to $G$ along each boundary component of $(V(G), A)$ and removing the interior of all vertices of $G$.

\begin{thm}\cite{16}\label{pro0}
Let $G$ be a ribbon graph and $A\subset E(G)$. Then
\begin{enumerate}
\item[(1)] $G^{\emptyset}=G$;
\item[(2)] $G^{E(G)}=G^*$, the geometric dual of $G$;
\item[(3)] $(G^A)^B=G^{A\bigtriangleup B}$, where $B\subset E(G)$ and $\bigtriangleup$ denotes the symmetric difference;
\item[(4)] $G$ is orientable if and only if $G^{A}$ is orientable.
\end{enumerate}
\end{thm}

The following lemma is no more than an observation.

\begin{Lemma} \label{seg} Let $G$ be a ribbon graph and $A\subset E(G)$. Let $G^A$ be the partial dual of $G$ with respect to $A$. Then
\begin{enumerate}
\item[(1)] there is a natural bijection between vertex line segments of $G$ and $G^A$; common line segments and edge line segments of edges inside $A$ exchange after taking the dual; and common line segments and edge line segments of edges outside $A$ keep unchanged after taking the dual.
\item[(2)] the boundary components of $G$ consist of vertex line segments and edge line segments; the boundary component of $G^A$ consist of vertex line segments, common line segments of edges inside $A$ and edge line segments of edges outside $A$.
\item[(3)] the boundary components of vertex discs of $G$ consist of vertex line segments and common line segments; the boundary components of vertex discs of $G^A$ (i.e. circles of the arrow presentation of $G^A$) consist of vertex line segments, common line segments of edges outside $A$ and edge line segments of edges inside $A$.
\item[(4)] there is a natural bijection between boundary components of $G$ and boundary components of $(V(G^A),A^c)$.
\end{enumerate}
\end{Lemma}

We use an example in Figure \ref{ls} to further illustrate the lemma. For (1), see (a) and (d). (b) and (c) are intermediate steps from (a) to (d). For (2), see (a) and (e). For (3), see (a) and (d) again. For (4), see (a) and (f).

\begin{figure}
\centering
\begin{minipage}{5cm}

\includegraphics[width=1.8in]{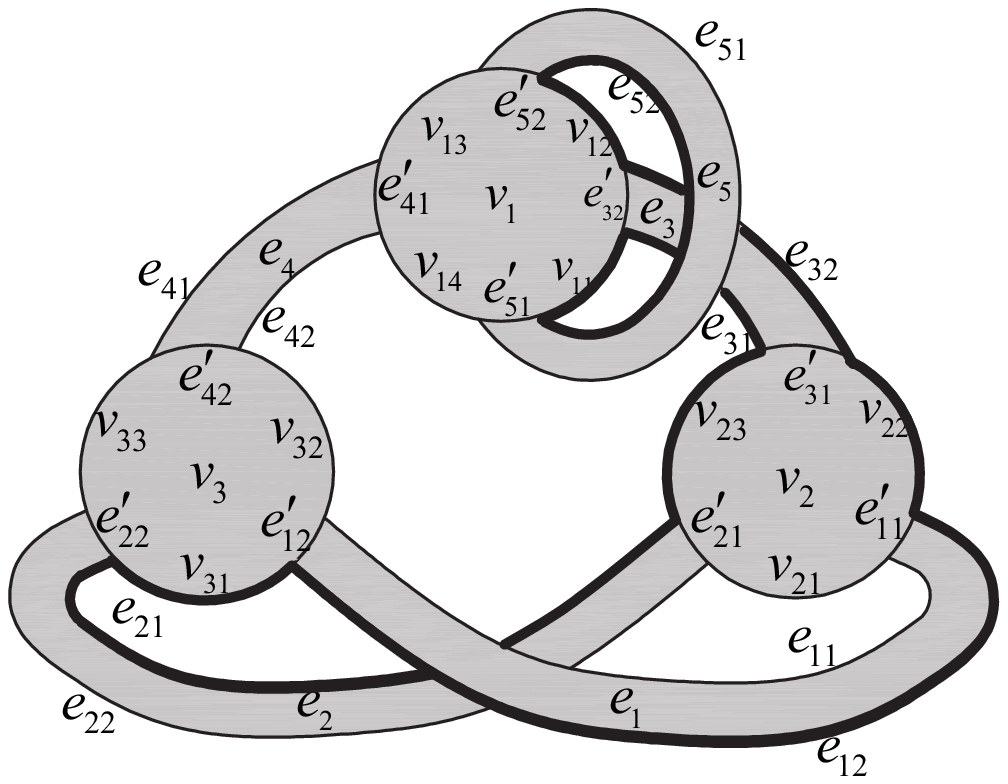}
\centerline{(a)}
\end{minipage}\begin{minipage}{5cm}

\includegraphics[width=1.8in]{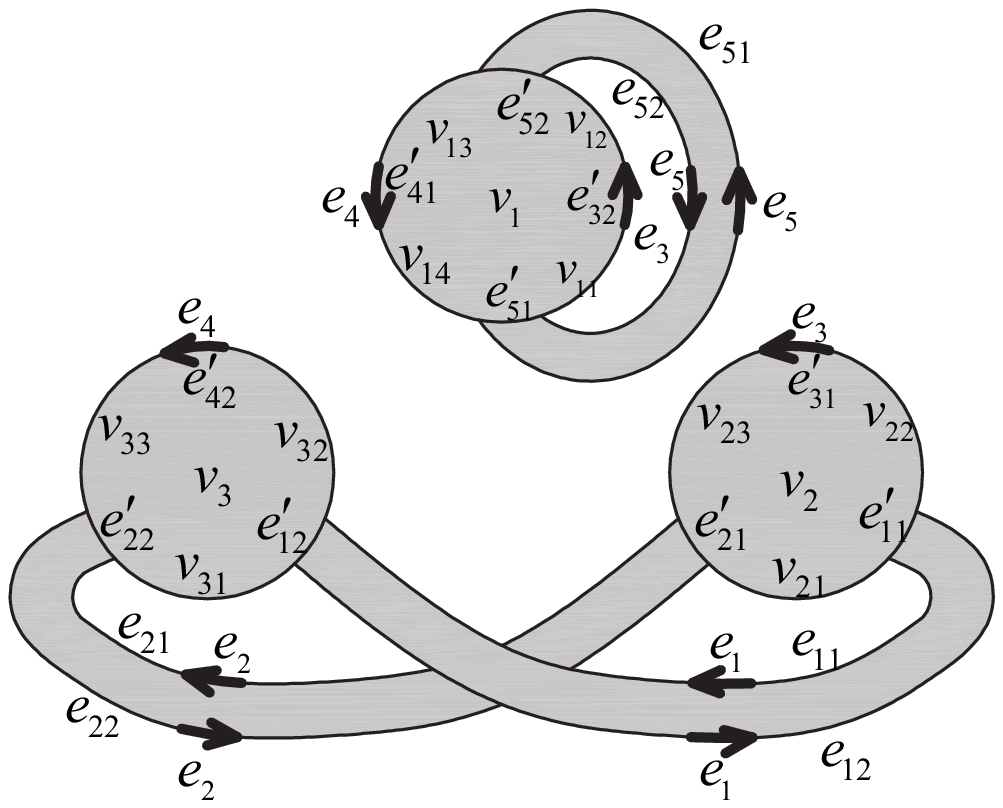}
\centerline{(b)}
\end{minipage}

\begin{minipage}{5cm}

\includegraphics[width=1.8in]{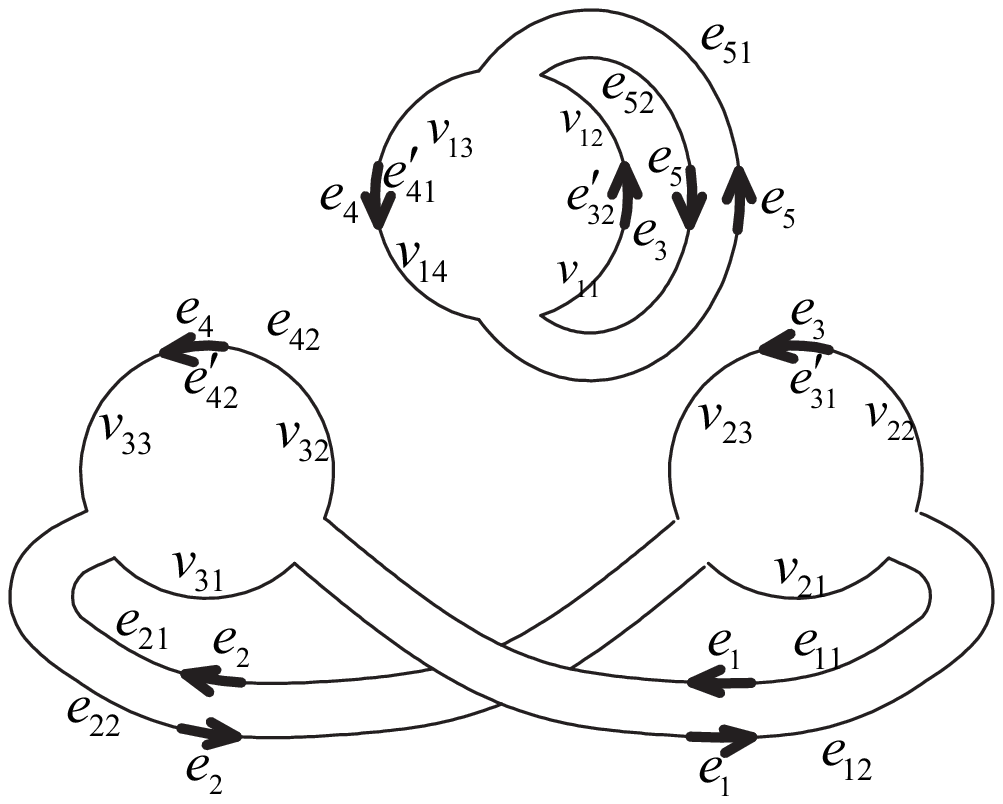}
\centerline{(c)}
\end{minipage}\begin{minipage}{5cm}

\includegraphics[width=1.8in]{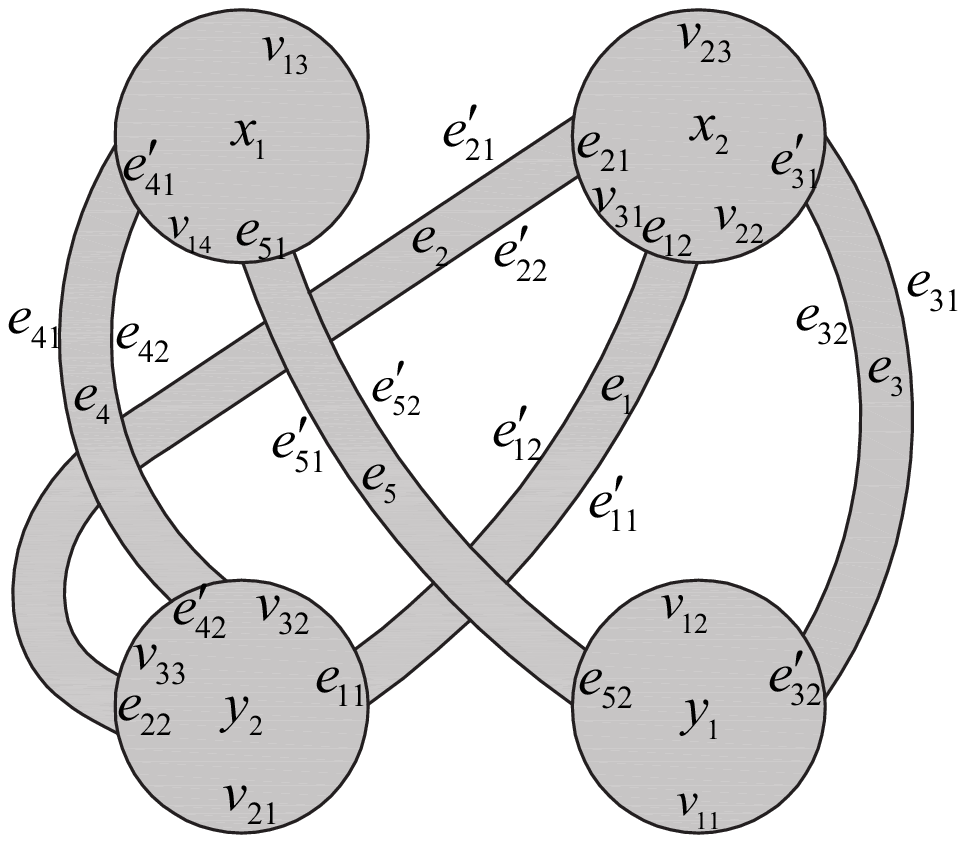}
\centerline{(d)}
\end{minipage}

\begin{minipage}{5cm}

 \includegraphics[width=1.8in]{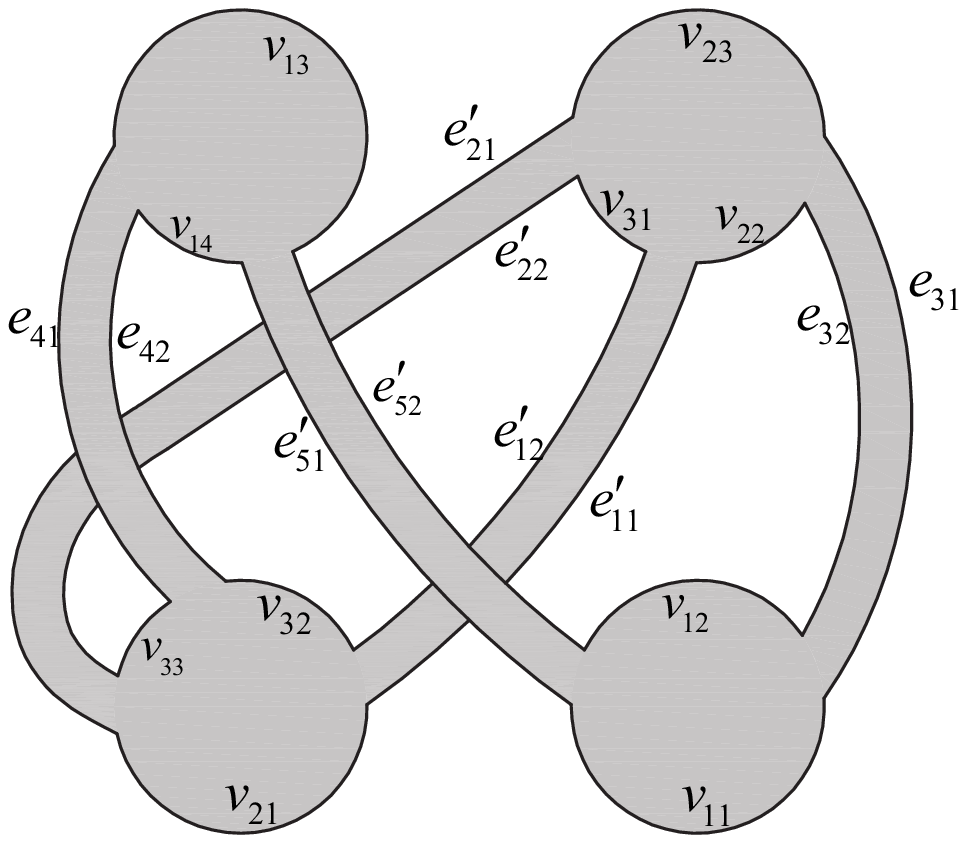}
\centerline{(e)}
\end{minipage}\begin{minipage}{5cm}

 \includegraphics[width=1.8in]{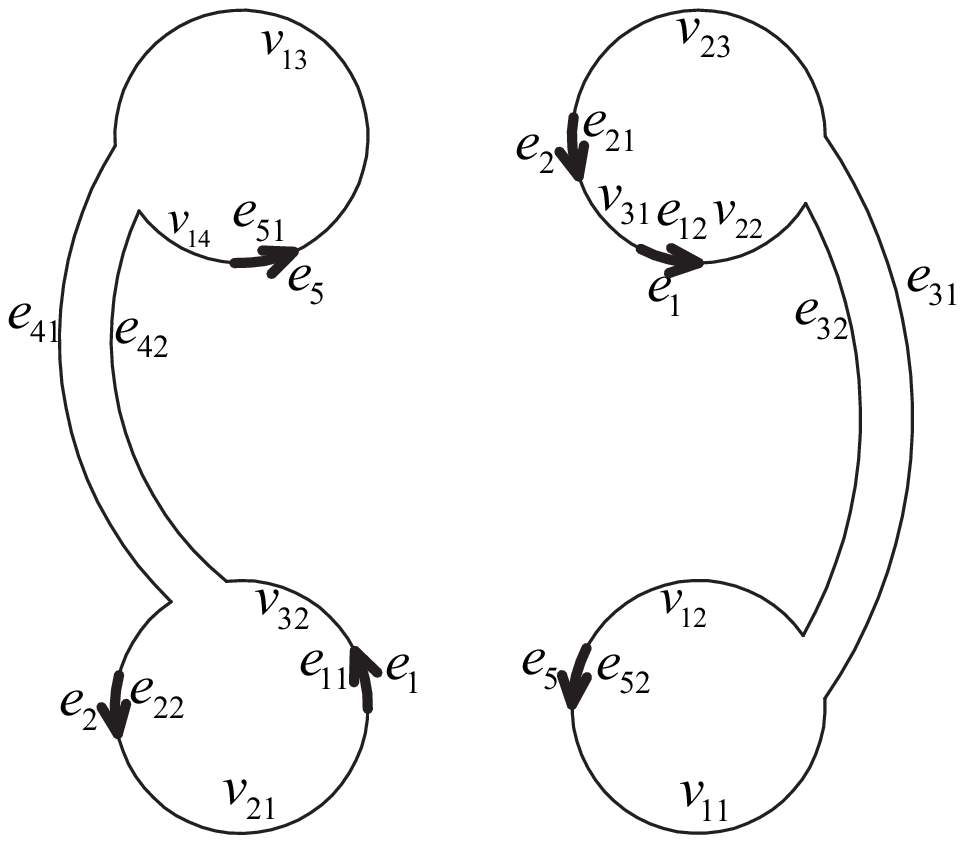}
\centerline{(f)}
\end{minipage}

\begin{minipage}{5cm}
\includegraphics[width=1.8in]{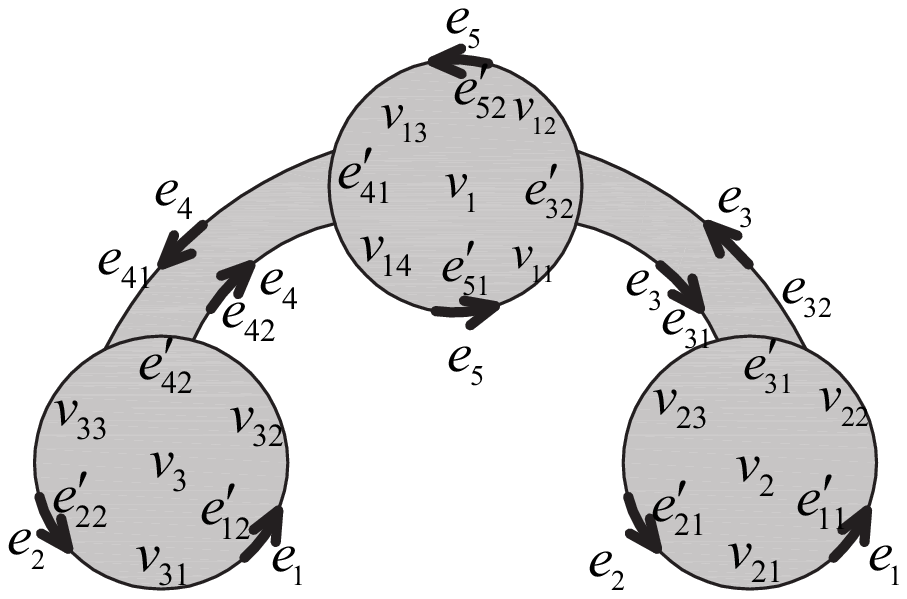}
\centerline{(g)}
\end{minipage}\begin{minipage}{5cm}

 \includegraphics[width=1.8in]{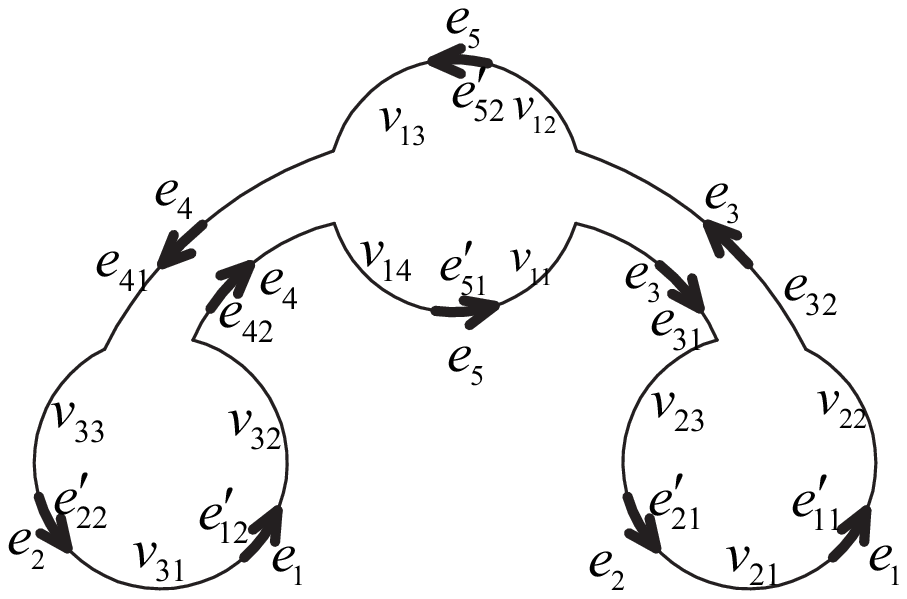}
\centerline{(h)}
\end{minipage}
\caption{Proofs of Lemma \ref{seg} and Theorem \ref{main}: (a) a ribbon graph $G$ and its common line segments, vertex line segments and edge line segments, (b) to (d) the construction of $G^A$ with $A=\{1,2,5\}$ and corresponding line segments, (e) the boundary component of $G^A$, (f) the
boundary components of $(V(G^A),A^c)$, (g) the marked spanning subgraph $(V(G),A^c)$ and (h) the arrow presentation of $G^{A^c}$.}\label{ls}
\end{figure}

\begin{thm}\label{main}
Let $G$ be a ribbon graph and $A\subset E(G)$. Then the partial dual $G^{A}$ is even-face graph
if and only if $G^{A^c}$ is Eulerian.
\end{thm}
\begin{proof}
By Theorem \ref{pro0}, $G^{A^c}$=$G^{A\triangle E(G)}=(G^A)^{E(G)}=(G^A)^*$. Take $A=E(G)$ in Lemma \ref{seg} (4), there is a bijection between the boundary components of a ribbon graph and vertex boundaries of its dual graph. Thus $G^A$ is even-face if and only if $G^{A^c}$ is Eulerian. In fact, there is a bijection between boundary components of $G^A$ and vertex boundaries of $G^{A^c}$ as shown in Figure \ref{ls} (e) and (h).
\end{proof}

\section{Medial graphs}
\noindent

Let $G$ be a cellularly embedded graph in $\Sigma$. We construct its medial graph $G_m$ by placing a vertex $v(e)$ on each edge $e$ of $G$ and then for each face $f$ with boundary $v_1e_1v_2e_2\cdots v_{d(f)}e_{d(f)}$, drawing the edges $(v(e_1),v(e_2))$,
$(v(e_2),v(e_3))$,$\cdots$,$(v(e_{d(f)}),v(e_1))$ non-intersected in a natural way inside the face $f$. Note that $G_m$ is also a cellularly embedded graph in $\Sigma$ and a 4-regular one. We shall denote by $t(G_m)$ the number of straight-ahead closed walks of $G_m$.

A checkerboard colouring of a cellularly embedded graph is an assignment of the
colour black or color white to each face such that adjacent faces receive different colours
(i.e., it is a proper 2-face colouring). A medial graph $G_m$ can always be checkerboardly coloured by colouring the faces
containing a vertex of the original graph $G$ black and the remaining faces white. We
call this checkerboard colouring of the medial graph $G_m$ the canonical checkerboard
colouring of $G_m$. See Figure \ref{mg} (a) and (b) for an example.

The medial graph of a ribbon graph or its arrow presentation is formed by translating it to the language of
cellularly embedded graph firstly, then forming the medial graph, and finally translating back. In Figure \ref{cdt}, we draw the medial graph of a ribbon graph as a cellularly embedded graph, not as a ribbon graph.

\subsection{Crossing-total directions}
\noindent

We are interested in particular in directed graphs which arise from orientations of
edges of medial graphs. A crossing-total direction of $G_m$ is an assignment
of an orientation to each edge of $G_m$ in such a way that for each vertex $v(e)$ of $G_m$, edges incident with $v(e)$ are "in,in,out,out", "in,in,in,in" or "out,out,out,out" oriented in cyclic order with respect to $v(e)$. See Figure \ref{mg} (c) for an example. If $G_m$ is equipped with the canonical checkerboard colouring and a fixed crossing-total direction, then we can partition
the vertices of $G_m$ into three classes, called $c$-vertices, $d$-vertices and $t$-vertices, respectively according
to the scheme shown in Figure \ref{cdt}. There is a natural bijection $e\leftrightarrow v(e)$ between edges of $G$ and vertices of $G_m$.
The corresponding edges of $G$ are called $c$-edges, $d$-edges and $t$-edges. See Figure \ref{mg} (d) for an example.

\begin{figure}[!htbp]
\begin{minipage}{3.2cm}
\includegraphics[width=1.2in]{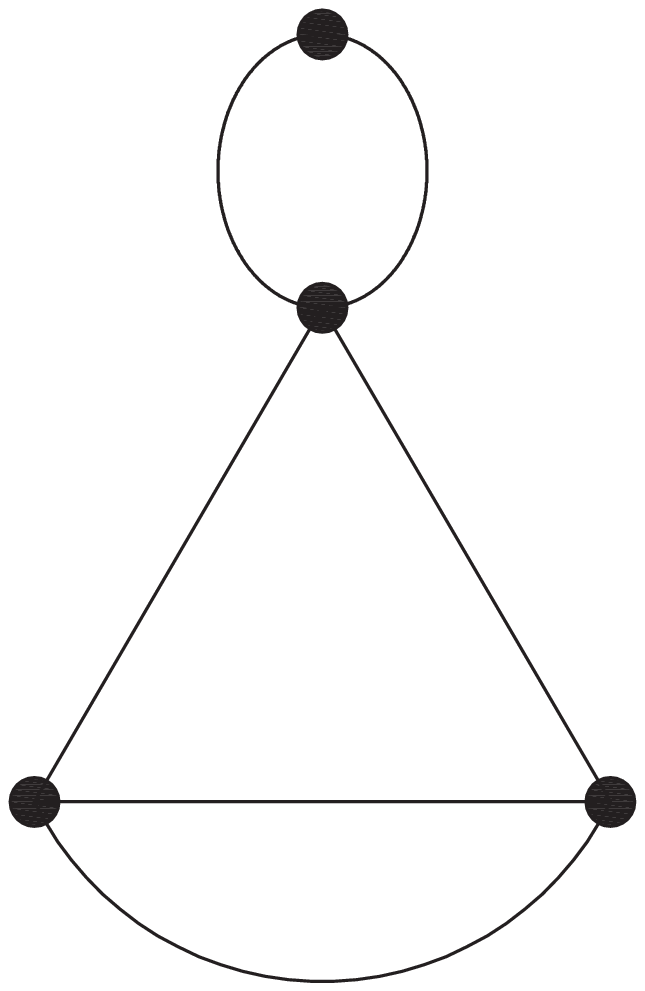}
\centerline{(a)}
\end{minipage}
\begin{minipage}{3.2cm}
\includegraphics[width=1.2in]{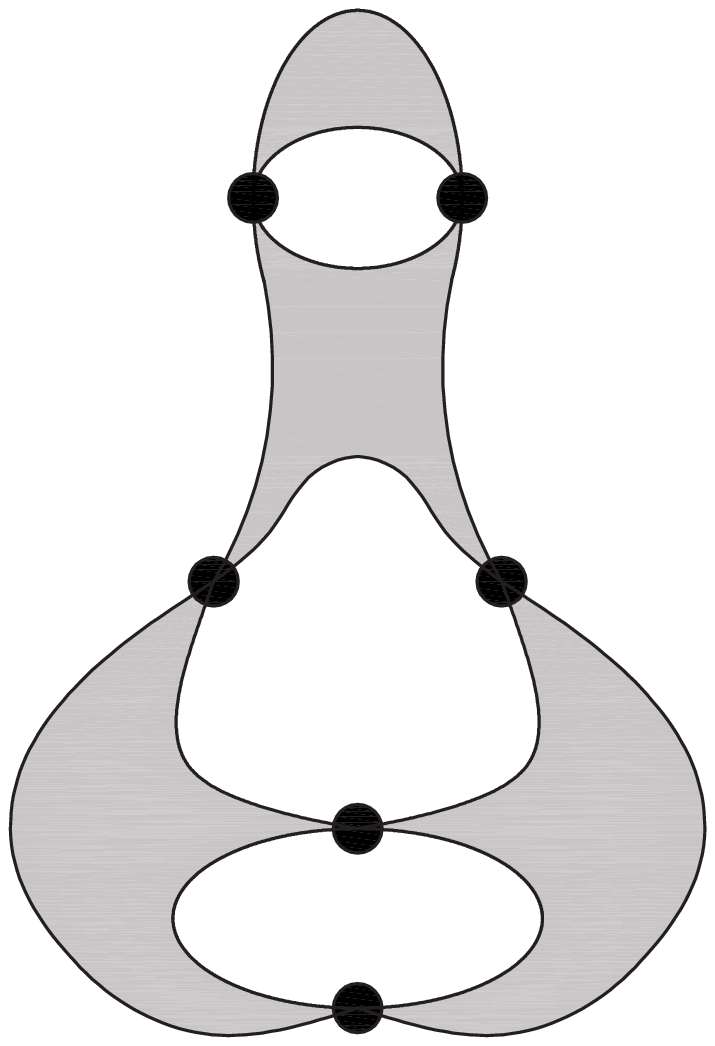}
\centerline{(b)}
\end{minipage}
\begin{minipage}{3.2cm}
\includegraphics[width=1.2in]{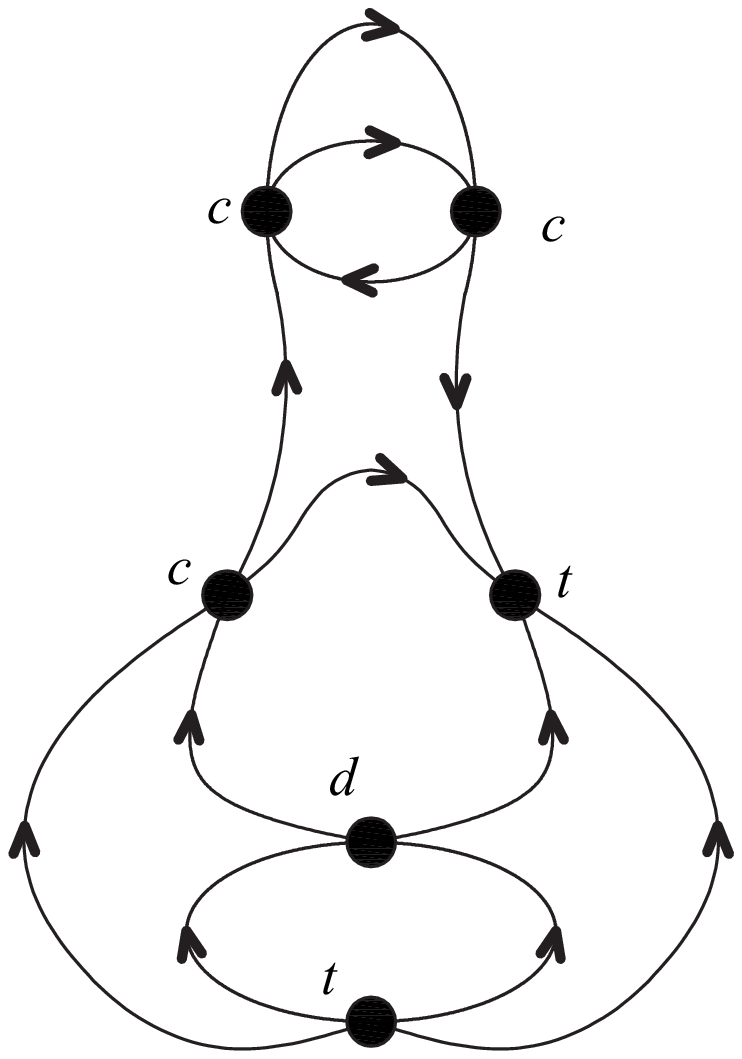}
\centerline{(c)}
\end{minipage}
\begin{minipage}{3.2cm}
\includegraphics[width=1.2in]{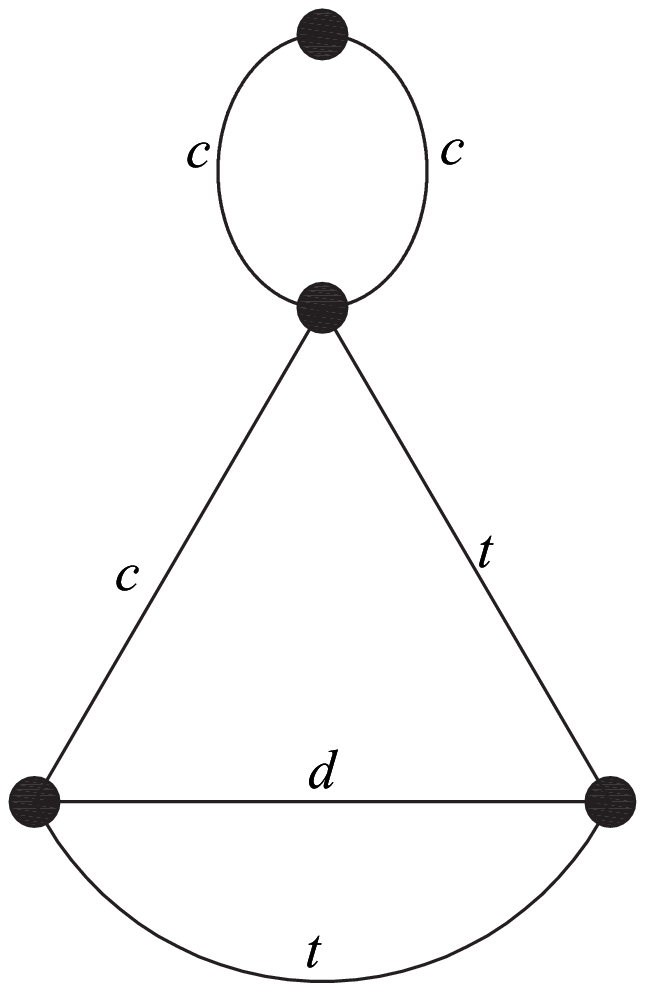}
\centerline{(d)}
\end{minipage}

\caption{(a) a plane graph $G$, (b) its medial graph $G_m$ with canonical checkerboard coloring and $t(G_m)=2$, (c) a crossing-total direction of $G_m$, and (d) the corresponding $\{c,d,t\}$-edges of $G$.}\label{mg}
\end{figure}

\begin{figure}[!htbp]
\centering
\includegraphics[width=5.5in]{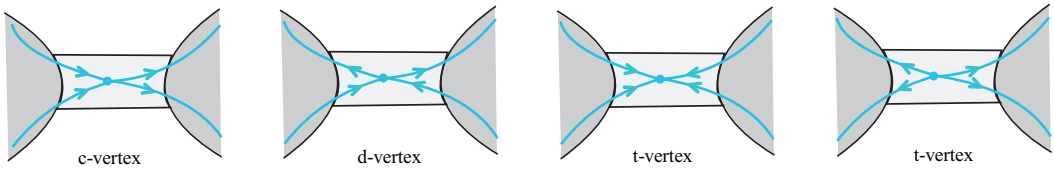}
\caption{{\footnotesize $\{c,d,t\}$-vertices.}}\label{cdt}
\end{figure}
Note that $G_m=(G^*)_m$ and the canonical checkerboard colouring of $(G^*)_m$ can be obtained from that of $G_m$ by switching colours black and white.

\begin{Lemma}\label{pro}
Let $G$ be a cellularly embedded graph and $e\in E(G)$. Under a fixed crossing-total direction of $G_m$, we have:
\begin{enumerate}
\item[(1)] $e$ is a $c$-edge in $G$ if and only if $e$ is a $d$-edge in $G^*$;
\item[(2)] $e$ is a $t$-edge in $G$ if and only if $e$ is a $t$-edge in $G^*$.
\end{enumerate}
\end{Lemma}

In \cite{HM}, Huggett and Moffatt introduced the notion of all-crossing direction of $G_m$ and then characterized all bipartite
partial duals of a plane graph using all-crossing directions. An all-crossing direction is a crossing-total direction without $t$-vertices.

\subsection{Graph states}

Let $G_m$ be a canonically checkerboard
coloured medial graph and $v(e)\in V(G_m)$. There are two ways to split $v(e)$ into two vertices of degree 2 as shown in Figure \ref{cc}. We call them white
smoothing and black smoothing respectively.

\begin{figure}[!htbp]
\centering
\includegraphics[width=4.5in]{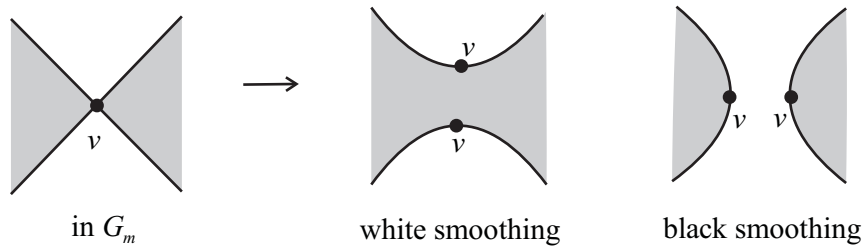}
\caption{{\footnotesize White and black smoothings.}}\label{cc}
\end{figure}

Let $G$ ba a cellularly embedded graph in $\Sigma$. A state $S$ of $G$ is a choice of white smoothing and black smoothing at each vertex of $G_m$. After splitting each vertex according to the state $S$, we obtain a set of disjoint closed cycles in $\Sigma$, call them state circles of $S$. We denote the number of state circles of $S$ by $c(S)$. A state is called even if each of its state circles is a cycle of even length. We denote the set of even states of $G$ by $\mathscr{ES}$.

Let $A\subset E(G)$ and $V_A=\{v(e)|e\in A\}\subset V(G_m)$. A state $S_A$ of $G$ associated with $A$ is the choice of choosing white smoothing for each vertex of $V_A$ and black smoothing for each vertex of $V(G_m)\backslash V_A$.

\begin{Lemma}\label{l1}
Let $G\subset\Sigma$ be a cellularly embedded graph and $A\subset E(G)$.
Let $S_A$ be the graph state of $G$ associated with $A$.
Then there is a bijection between the vertices (of degree 2) of state circles of $S_A$ and the marking arrows (ignoring directions) of the arrow representation of $G^{A}$.
\end{Lemma}

\begin{proof}
Compare Figure \ref{cc} with Definition \ref{dpd}.
\end{proof}

\begin{Lemma}\label{map}
Let $N_{CT}(G)$ be the number of crossing-total directions of $G_m$. Then
\begin{eqnarray}
2^{t(G_m)}\leq N_{CT}(G)\leq \sum_{S\in \mathscr{ES}}2^{c(S)}.
\end{eqnarray}
\end{Lemma}
\begin{proof}
Recall that $t(G_m)$ is the number of straight-ahead closed walks of $G_m$. For each such closed walk, one can choose clockwise or anticlockwise orientation, then obtaining an all-crossing direction of $G_m$. Conversely using an all-crossing direction of $G$ one can split $E(G_m)$ into straight-ahead closed walks and such walks are either clockwise oriented or anticlockwise oriented. Thus the number of all-crossing directions of $G_m$ is exactly $2^{t(G_m)}$. This is in fact pointed out in Section 4.1 of \cite{HM}. Any all-crossing direction of $G_m$ is also a crossing-total direction and hence the lower bound holds.

Now we prove the upper bound. Let $\mathcal{O}$ be a crossing-total direction of $G_m$. We shall take a white smoothing for each $d$-edge and a black smoothing for each $c$-edge, take a white or black smoothing for each $t$-edge of $G$, then we obtain a state $S_{\mathcal{O}}$ and for each state circle, its orientation coming from $\mathcal{O}$ is alternating between clockwise and anticlockwise. Thus $S_{\mathcal{O}}\in \mathscr{ES}$ and is equipped with an alternating orientation.
\end{proof}

\section{Main results}
\noindent

In \cite{HM}, Huggett and Moffatt gave a sufficient and necessary condition for partial duals of a plane graph $G$ to be bipartite in terms all-crossing directions of its medial graph $G_m$. Roughly speaking, it is obtained by partitioning $E(G)$ into $A$ and $A^c$ and considering $(V(G),A)$ and $(V(G^*),A^c)$. In this section we use a direct method to extend their result from plane graphs to orientable ribbon graphs.

\begin{thm}\label{main0}
Let $G$ be an orientable ribbon graph and $A\subset E(G)$. Then $G^{A}$ is bipartite
if and only if $A$ is the set of $c$-edges arising from an all-crossing direction
of $G_m$.
\end{thm}

\begin{proof} ($\Longrightarrow$) Let $(X,Y)$ be the bipartite partition of $V(G^A)$. Suppose that $X=\{x_1,x_2,\cdots,x_k\}$ and $Y=\{y_1,y_2,\cdots,y_l\}$. Note that circles of arrow representation of $G^A$ are exactly boundaries of vertex discs of $G^A$. By Lemma \ref{seg} (3), vertex line segments of $G$ all belong to boundaries of $x_i$ or $y_j$. We rename the vertex line segment as follows: if the vertex line segment $v_{rs}$ on the boundary of $x_i$, it will be renamed $x_i$; if it is on the boundary of $y_j$, it will be renamed $y_j$. An example is given in Figure \ref{bp} (a), where $X=\{x_{1},x_{2}\}$, $Y=\{y_1,y_2\}$, $\partial(x_1)\supset \{v_{14},v_{13}\}$, $\partial(y_1)\supset \{v_{11},v_{12}\}$, $\partial(x_2)\supset \{v_{22},v_{23},v_{31}\}$, $\partial(y_2)\supset \{v_{21},v_{33},v_{32}\}$. It is clear that on each boundary component of $G^A$ $x_i$'s and $y_j$'s alternate since $G^A$ is bipartite. Recall that there is an one-to-one correspondence between boundary components of $G^A$ and boundaries of vertex discs (i.e. circles of arrow presentation) of $G^{A^c}$. Thus $x_i$'s and $y_j$'s alternate in circles of arrow representation of $G^{A^c}$. An example is given in Figure \ref{bp} (b).

\begin{figure}
\centering
\begin{minipage}{7cm}
\includegraphics[width=2.6in]{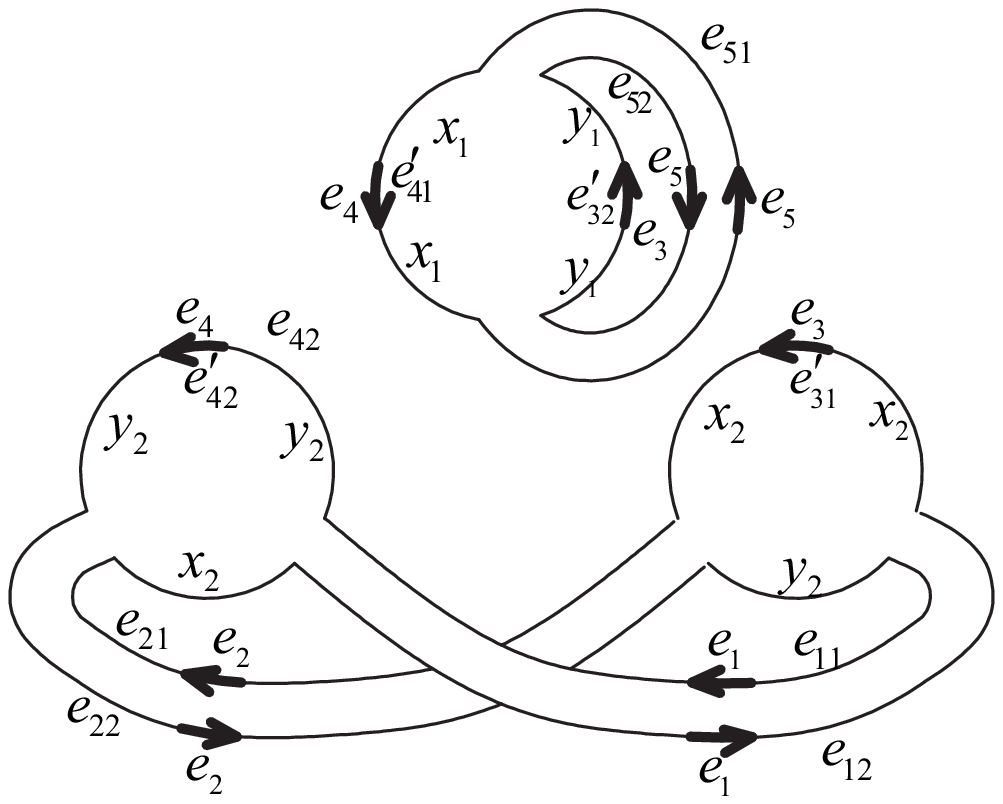}
\centerline{(a)}
\end{minipage}\begin{minipage}{7cm}
\includegraphics[width=2.6in]{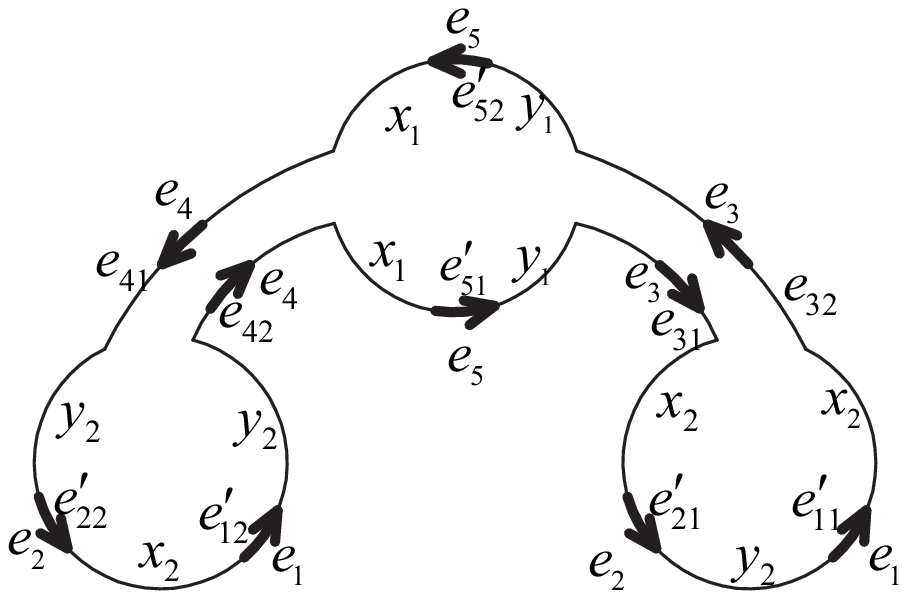}
\centerline{(b)}
\end{minipage}
\caption{An example of (a) renaming vertex line segments of $G$ in $G^A$, (b) boundaries of vertex discs of $G^{A^c}$.}\label{bp}
\end{figure}

The local depiction near an edge $e_h$ of arrow presentation of $G^A$ is given in Figure \ref{c1} (a) and (a') for $e_h\in A$ and $e_h\notin A$, respectively.
Figure \ref{c1} (b) and (b') are their counterparts of arrow presentation of $G^{A^c}$.

Note that $G$ is orientable implies that $G^{A^c}$ is orientable by Theorem \ref{pro0} (4). Hence by Lemma \ref{pro1} directions of arrows in the same circle of arrow presentation of $G^{A^c}$ can have same orientation as shown in Figure \ref{bp} (b), and if one of a pair of marking arrows with the same label is directed from $x_i$ to $y_j$ then the other will be directed from $y_j$ to $x_i$. See Figure \ref{c1} (b) and (b').

If we replace marking arrow directed from $x_i$ to $y_j$ with a 2 valent vertex with two ingoing half edges, replace marking arrow directed from $y_j$ to $x_i$ with a 2 valent vertex with two outgoing half edges. Note that the orientations of two-half edges of $x_i$ or $y_j$ coincide, merging into one arrow. Thus we obtain an alternating orientation on circles of arrow representation of $G^{A^c}$. See Figure \ref{c1} (c) and (c'). Transfer them to orientations of $G_m$ as shown in Figure \ref{c1} (d) and (d'), we obtain an all-crossing direction of $G_m$. Furthermore each edge of $A$ is a $c$-edge and each edge of $A^c$ is a $d$-edge as shown in Figure \ref{c1} (e) and (e').

($\Longleftarrow$) Let $\mathcal{O}$ be an all-crossing direction of $G_m$ such that $A$ is the set of $c$-edges arising from $\mathcal{O}$. Then $A^c$ is the set of $d$-edges. We label $x$ the edges of $G_m$ if the black colour is on their right along the directions,  $y$ the edges of $G_m$ if the black colour is on their left along the directions. See Figure \ref{c1} (d) and (d'). Let $S_A$ be the state of $G$ associated with $A$. State circles of $S_A$ correspond to circles of arrow presentation of $G^A$ as shown in Figure \ref{c1} (a) and (a'). Then circles labelled $x$ and circles labelled $y$ form vertex bipartition of $G^A$. Thus $G^A$ is bipartite.
\begin{figure}[!htbp]
\begin{minipage}{8cm}
\includegraphics[width=5.2in]{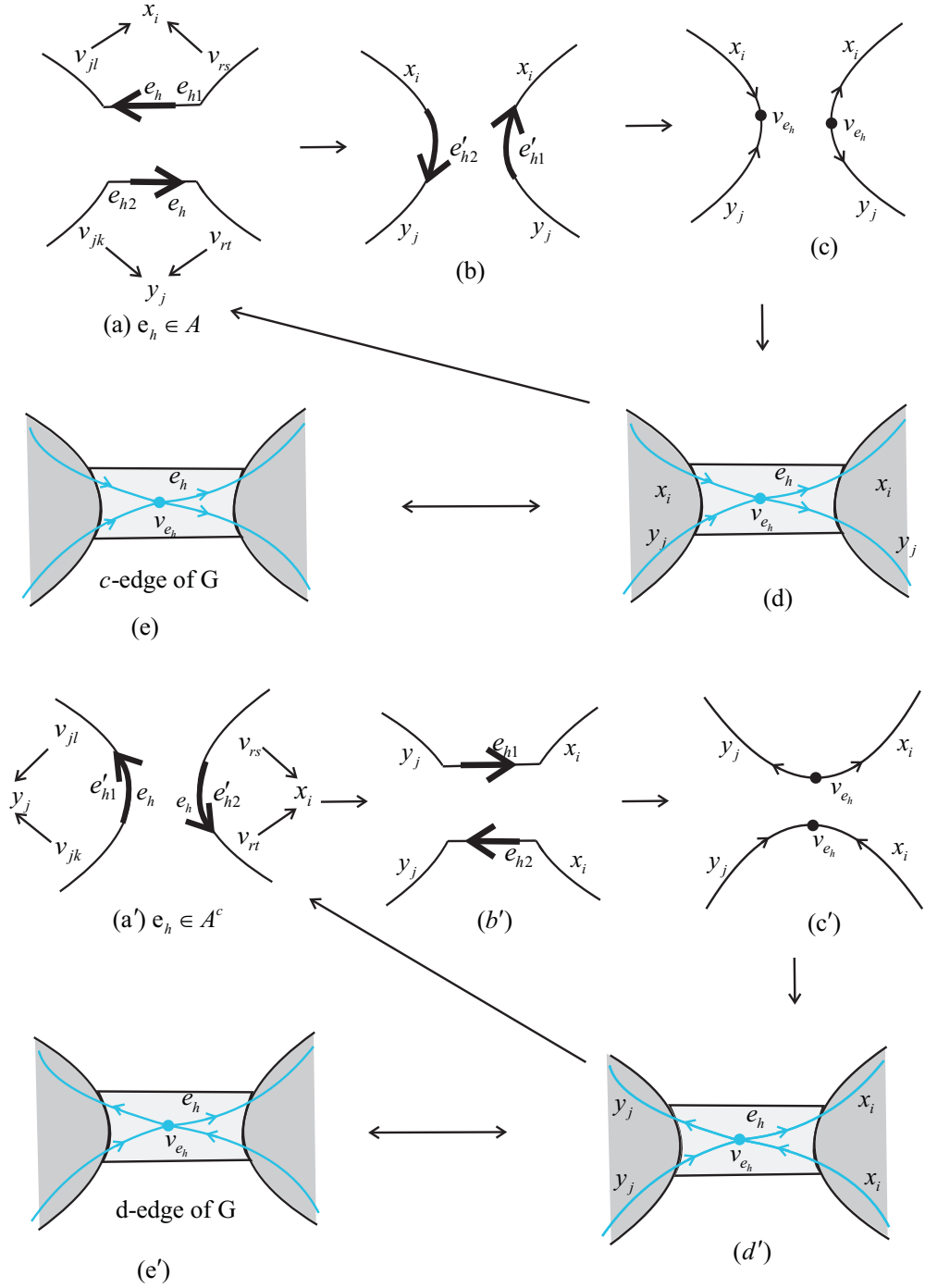}
\end{minipage}
\caption{{\footnotesize Proof of Theorem \ref{main0}.}}\label{c1}
\end{figure}

\end{proof}

As a direct consequence, we have

\begin{cor}\cite{HM}\label{co}
Let $G$ be a plane graph and $A\subset E(G)$. Then $G^{A}$ is bipartite
if and only if $A$ is the set of $c$-edges arising from an all-crossing direction
of $G_m$.
\end{cor}

\begin{remark}
The condition 'orientable' in Theorem \ref{main0} is necessary. In Figure \ref{cx} (a), we give a counterexample $G$, which is non-orientable. The medial graph $G_m$ is given in Figure \ref{cx} (b) and its unique $c$ and $d$-vertex partition is shown in Figure \ref{cx} (c). The partial dual of $G$ with respect to the set of all $c$-edges is a graph having a loop (the edge labeled 1) as shown in Figure \ref{cx} (4), hence not bipartite. It is an open question to characterize bipartite partial duals of non-orientable ribbon graphs.

\begin{figure}[!htbp]
\begin{minipage}{7cm}
\includegraphics[width=2.6in]{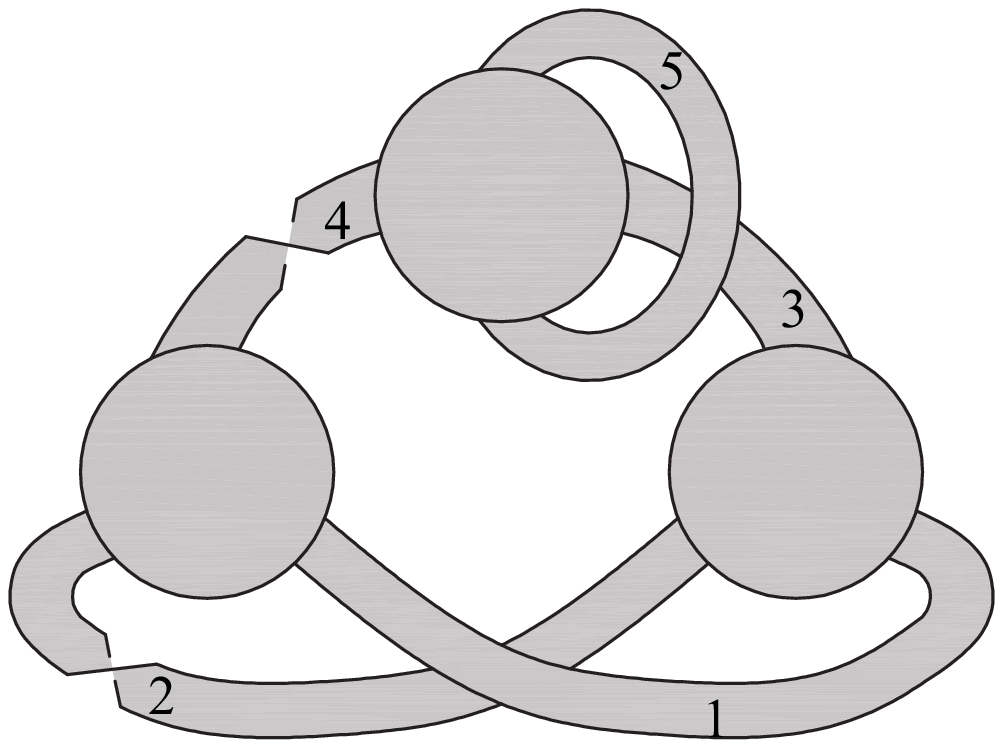}
\centerline{(a)}
\end{minipage}\begin{minipage}{7cm}
\includegraphics[width=2.6in]{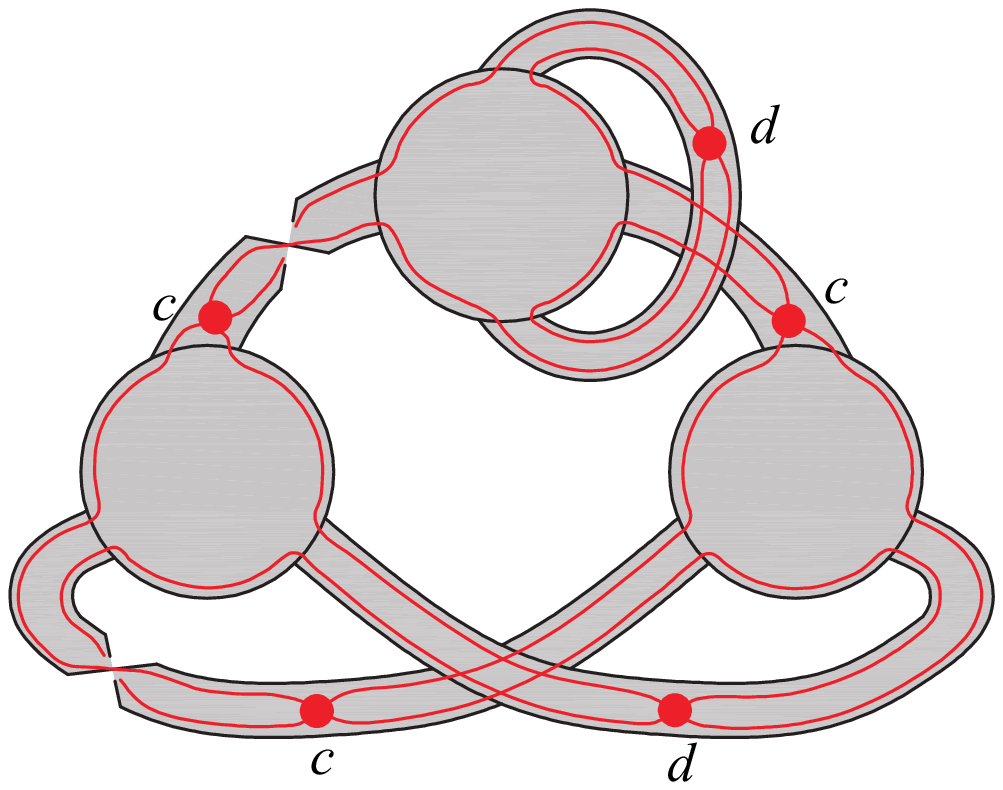}
\centerline{(b)}
\end{minipage}

\begin{minipage}{7cm}
\includegraphics[width=2.6in]{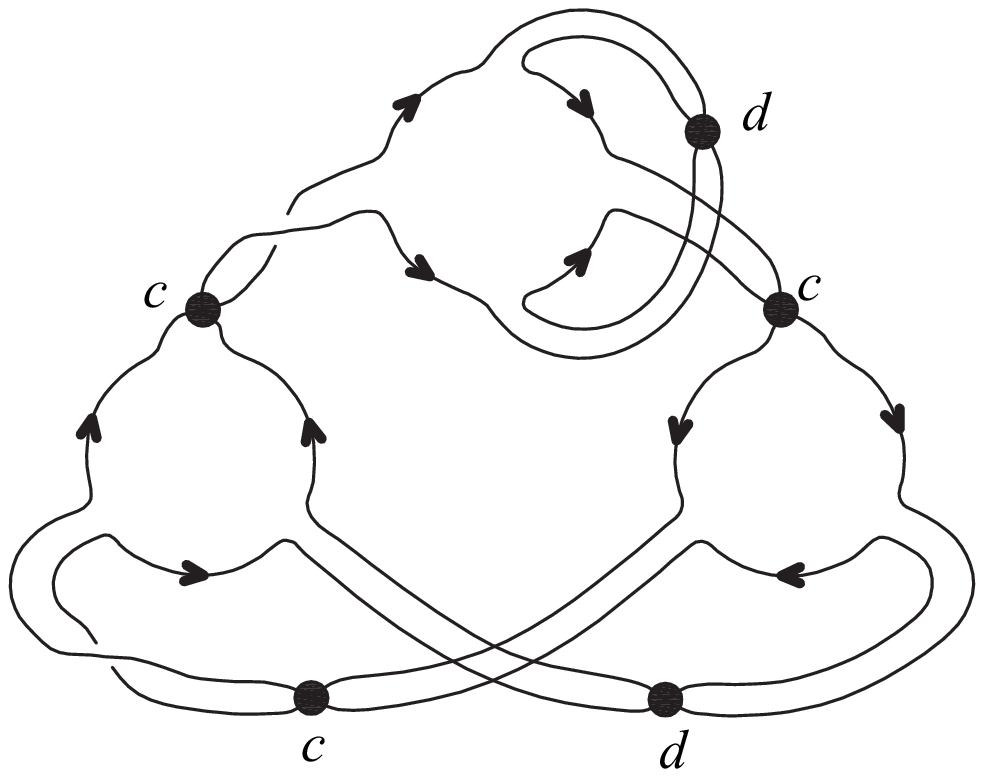}
\centerline{(c)}
\end{minipage}\begin{minipage}{7cm}
\includegraphics[width=2.6in]{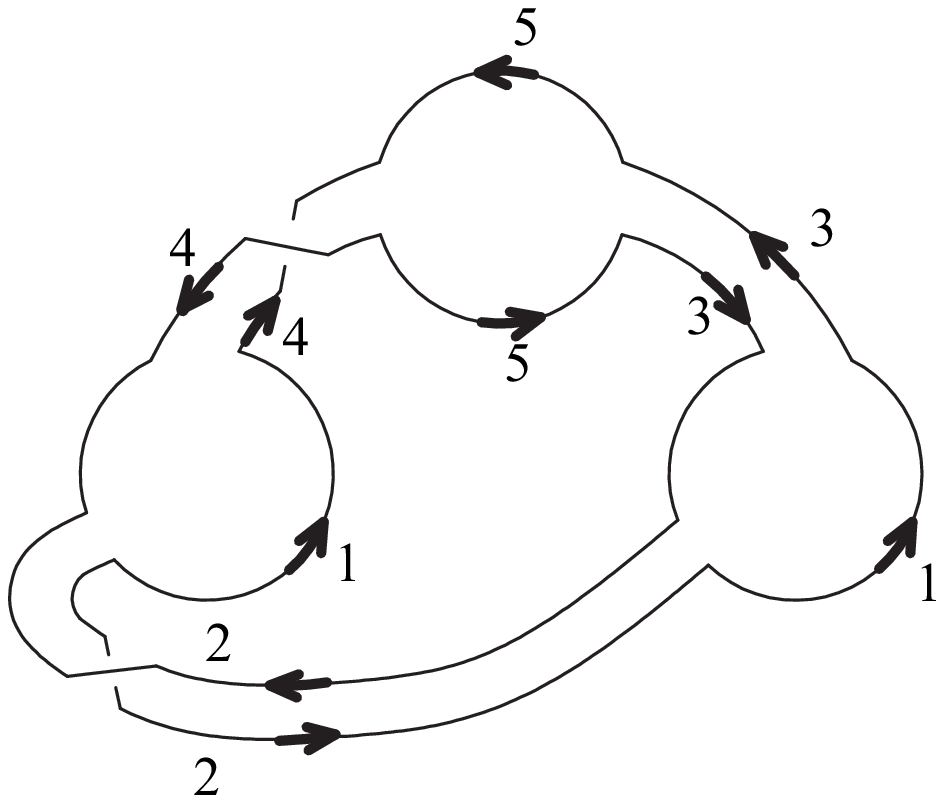}
\centerline{(d)}
\end{minipage}
\caption{A counterexample: (a) a non-orientable ribbon graph $G$, (b) its medial graph $G_m$ (embedded in $G$) with $t(G_m)=1$, (c) its unique $\{c,d\}-$partitions of $V(G_m)$ and (d) arrow presentation of $G^{\{2,3,4\}}$.}\label{cx}
\end{figure}

\end{remark}

\begin{Lemma}\label{l2}
Let $G$ be a ribbon graph and $A\subset E(G)$. Then $G^A$ is Eulerian
if and only if $S_A$ of $G$ associated with $A$ is even.
\end{Lemma}
\begin{proof}
The degree of a vertex of $G^{A}$ is equal to the number of marking arrows of the corresponding circle of arrow presentation of $G^A$. By Lemma \ref{l1}, it is exactly the length of corresponding state circle of $S_A$ of $G$. Thus $G^A$ is Eulerian if and only if $S_A$ is even.
\end{proof}

We now state and prove our second main but simple result, a characterization
of Eulerian partial duals of a ribbon graph in terms of crossing-total directions of its medial graph.

\begin{thm}\label{main1}
Let $G$ be a ribbon graph and $A\subset E(G)$. Then $G^{A}$ is Eulerian
if and only if $A$ is the union of the set of all $d$-edges and the set of some $t$-edges arising from a crossing-total direction
of $G_m$.
\end{thm}
\begin{proof}
($\Longrightarrow$) By Lemma \ref{l2}, $S_A$ is even. So we can give an alternating (clockwise and counter-clockwise) orientation to each circle of $S_A$. Transferring them to $G_m$, we obtain a crossing-total direction of $G_m$. It is easy to see that each edge of $A$ is either a $d$-edge or a $t$-edge and each edge of $A^c$ is either a $c$-edge or a $t$-edge. Thus $A$ is the union of the set of all $d$-edges and some $t$-edges arising from a crossing-total direction
of $G_m$.

($\Longleftarrow$) Conversely, if $A$ is the union of the set of all $d$-edges and some $t$-edges arising from a crossing-total direction
of $G_m$. Then $S_A$ is an even state of $G$ since the edge orientations of each state circle of $S_A$ are alternating between clockwise and counterclockwise orientation. Thus $G^{A}$ is Eulerian by Lemma \ref{l2}.
\end{proof}

Using Theorem \ref{main} and Lemma \ref{pro}, we obtain:

\begin{cor}\label{main2}
Let $G$ be a ribbon graph and $A\subset E(G)$. Then $G^{A}$ is even-face graph
if and only if $A$ is the union of the set of all $c$-edges and the set of some $t$-edges arising from a crossing-total direction
of $G_m$.
\end{cor}

\section*{Acknowledgements}
\noindent

This work is supported by NSFC (No. 11671336) and President's Funds of Xiamen University (No. 20720160011).


\section*{References}

\begin{thebibliography}{10}

\bibitem{Bon} J. A. Bondy, U. S. R. Murty, Graph theory, Graduate Texts in Mathematics 244,
Springer, Berlin, 2008.


\bibitem{16} S. Chmutov, Generalized duality for graphs on surfaces and the signed Bollob\'{a}s-Riordan polynomial, J Combin. Theory Ser. B 99 (2009) 617-638.

\bibitem{Chmu} S. Chmutov, I. Pak, The Kauffman bracket of virtual
links and the Bollob\'as-Riordan polynomial, Mosc. Math. J. 7
(2007) 409-418.

\bibitem{Chmu2} S. Chmutov, J. Voltz, Thistlethwaite's theorem for
virtual links, J. Knot Theory Ramifications 17 (2008) 1189-1198.

\bibitem{Lin1} O. Dasbach, D. Futer, E. Kalfagianni, X.-S. Lin and N.
Stoltzfus, The Jones polynomial and graphs on surfaces, J. Combin.
Theory, Ser. B 98 (2008) 384-399.


\bibitem{Mo} J. A. Ellis-Monaghan, I. Moffatt, Graphs on Surfaces: Dualities, Polynomials,
and Knots, Springer Briefs in Mathematics, 2013.

\bibitem{HM} S. Huggett, I. Moffatt, Bipartite partial duals and circuits in medial graphs, Combinatorica 33 (2013) 231-252.


\bibitem{Metrose} M. Metsidik, Characterization of Some Properties of Ribbon
Graphs and Their Partial Duals, PhD thesis, Xiamen University, 2017.

\bibitem{MJ} M. Metsidik, X. Jin, Eulerian partial duals of plane graphs, J. Graph Theory, accepted.

\end{thebibliography}
\end{document}